\documentclass[11pt]{article}

\usepackage[T1]{fontenc}
\usepackage{amsmath,amssymb,amsthm,mathtools}

\usepackage{lmodern}
\usepackage[a4paper,top=28mm,bottom=29mm,left=29mm,right=29mm]{geometry}
\usepackage{microtype}
\usepackage{needspace}
\usepackage[numbers,sort&compress]{natbib}
\usepackage[hidelinks]{hyperref}

\hypersetup{
  pdftitle={On a conjecture on Romanoff type representations},
  pdfauthor={Xiao-Hui Yan},
  pdfsubject={On a conjecture on Romanoff type representations},
  pdfkeywords={Romanoff type problem, prime, asymptotic density, sieve methods}
}

\newtheorem{theorem}{Theorem}[section]
\newtheorem{lemma}[theorem]{Lemma}
\newtheorem{corollary}[theorem]{Corollary}
\theoremstyle{remark}

\theoremstyle{plain}
\numberwithin{equation}{section}

\newcommand{\N}{\mathbb N}
\newcommand{\Z}{\mathbb Z}
\newcommand{\PP}{\mathcal P}
\newcommand{\A}{\mathcal A}

\allowdisplaybreaks[1]
\title{On a conjecture on Romanoff type representations}
\author{Xiao-Hui Yan\\[4pt]
  \small School of Mathematics and Statistics, Anhui Normal University,\\[-1pt]
  \small Wuhu 241002, P.\ R.\ China\\[3pt]
  \small\texttt{yanxiaohui\_1992@163.com}}
\date{}

\begin{document}
\maketitle

\begin{abstract}
Let $r_1,\ldots,r_t$ be positive integers, and let $\A$ be the set of integers
of the form $2^{k_1^{r_1}}+\cdots+2^{k_t^{r_t}}$, where
$k_1,\ldots,k_t$ are positive integers. For any positive integer $n$, let $r(n)$ be the
number of representations $n=p+a$ with $p$ an odd prime and $a\in \A$.
We prove that if $r_1^{-1}+\cdots+r_t^{-1}=1$, then, for every fixed
positive integer $m$, the set of integers $n$ with $r(n)\ge m$ has positive
lower asymptotic density. In particular, this confirms a conjecture of
Chen.
\end{abstract}


\section{Introduction}
Let $\PP$ be the set of all odd primes and $\N$ be the set of all positive
integers. For a set $B\subseteq\N$, let
\[
 B(x)=|B\cap[1,x]|.
\]

In 1849, de Polignac \cite{Polignac} conjectured that every odd number greater than $3$ can be
represented as the sum of an odd prime and a power of $2$. In 1934, Romanoff~\cite{Romanoff} proved that the set of positive integers of the form $p+2^k$ with $p\in\PP$ and $k\in\N$ has positive lower asymptotic density. This theorem provides a classical example of two sets of
asymptotic density zero whose sumset has positive lower asymptotic density. In 1950, Erd\H{o}s~\cite[Theorem~3]{Erdos} showed that there exists an arithmetic progression of positive odd integers such that none of them are of the form $p+2^k$, and so the set of positive odd integers which are not of the form $p+2^k$ also has positive lower asymptotic density.
Later, there are some quantitative results on density; see, for example, \cite{ChenSun,HabsiegerRoblot,Pintz,ElsholtzSchlagePuchta}.

In addition, Erd\H{o}s~\cite[Theorem~1]{Erdos} proved that the number of representations of an integer as $p + 2^k$ is unbounded. Chen and Ding~\cite[Corollary~4]{ChenDingErdos} extended this result by showing that the number of
representations of an integer as $p+a$, with $a$ in an infinite set of integers,
is unbounded. Later, they \cite{ChenDingQuantitative} obtained a quantitative result.

In 2010, Chen~\cite{ChenSparse} studied integers of the form $p+2^k$ with $k$
in a subset of $\N$. Elsholtz, Luca and
Planitzer~\cite[Theorems~1 and~2]{ElsholtzLucaPlanitzer} studied integers of the form $p+2^{2^k}+j!$ and $p+2^{2^k}+2^q$, and obtained positive lower density results, where
$j,k$ are nonnegative integers and $p,q$ are primes. Ding~\cite{DingSquares} proved that the set
\[
 \{p+2^{a^2}+2^{b^2}:p\in\PP,\ a,b\in\N\}
\]
has positive lower asymptotic density, answering a question of Yang and Chen \cite{YangChen}.

Let $r_1,\ldots,r_t$ be fixed positive integers, and let
\[
 \A=\left\{2^{k_1^{r_1}}+\cdots+2^{k_t^{r_t}}:
                 k_1,\ldots,k_t\in\N\right\},
 \qquad
 \sigma=\frac1{r_1}+\cdots+\frac1{r_t}.
\]
For any positive integer $n$, let
\[
 r(n)=|\{a\in \A:n-a\in\PP\}|.
\]
It is clear that $r(n)=0$ for every even integer $n$.

Chen and Xu~\cite[Corollary 1.2]{ChenXu} proved that $\PP+\A$ has a positive lower
asymptotic density if $\sigma\ge1$. If $\sigma<1$, then $\PP+\A$ has
asymptotic density zero. Ding and Zhai~\cite{DingZhai} proved the same result with $k_i^{r_i}$ replaced
by $\lfloor k_i^{r_i}\rfloor$ for any positive real numbers $r_1,\ldots,r_t$.

Recently, Chen~\cite[Corollary 1.1]{ChenII} proved that, if $\sigma>1$, then for every positive
integer $m$, the set of positive odd integers $n$ with $r(n)\ge m$ has a positive lower
asymptotic density. In the case $\sigma=1$, he~\cite[Theorem 1.3]{ChenII} proved the same result
for $m=2$ under the additional assumption that $r_1,\ldots,r_t$ are
pairwise distinct. Chen~\cite[Conjecture~1.2]{ChenII} conjectured that
this assumption can be removed. He also noted that the case $r_1=1$
is already covered by his Theorem~1.3 or by a result of Elsholtz and
Schlage-Puchta~\cite[Theorem 3]{ElsholtzSchlagePuchta}.

In this paper, we prove the following result.

\Needspace{10\baselineskip}
\begin{theorem}\label{thm:main}
If $r_1,\ldots,r_t$ are positive integers with $r_1^{-1}+\cdots+r_t^{-1}=1$, then for every positive integer $m$, the set of positive odd integers $n$ with $r(n)\ge m$ has a positive lower asymptotic density.
\end{theorem}

Taking $m=2$ in Theorem~\ref{thm:main} confirms Chen's conjecture.
The following corollary follows immediately from Theorem \ref{thm:main},
\cite[Corollary~1.1]{ChenII}, and the fact that $\PP+\A$ has asymptotic density zero if
$\sigma<1$.

\begin{corollary}\label{cor:criterion}
Let $r_1,\ldots,r_t$ and $m$ be positive integers. Then the set of positive odd integers $n$ with $r(n)\ge m$ has a positive lower asymptotic density if and only if $ r_1^{-1}+\cdots+r_t^{-1}\ge1$.
\end{corollary}

\section{Preliminary lemmas}

Throughout this section, $p$ always denotes a prime. Let
\[
 w_s=|\{i:1\le i\le t,\ r_i=s\}|, \qquad w=\prod_{s\ge1}w_s!.
\]

\begin{lemma}\cite[Lemma~2.1]{ChenII}\label{lem:Acount}
Let $r_1,\ldots,r_t$ be positive integers. Then, as $x\to+\infty$, we have
\[
 \A(x)=\left(1+o(1)\right)\frac{1}{w}(\log 2)^{-\frac{1}{r_1}-\cdots-\frac{1}{r_t}}(\log x)^{\frac{1}{r_1}+\cdots+\frac{1}{r_t}}.
\]
\end{lemma}

\begin{lemma}\cite[equation (3.17)]{DingLi}\label{lem:pair}
Assume that $\sigma=r_1^{-1}+\cdots+r_t^{-1}=1$. Let $v$ be a positive integer. Then, as $x\to+\infty$, we have
\[
 \sum_{\substack{b<b'\le x\\b,b'\in\A}}
 \prod_{\substack{p\mid(b-b')\\2<p\le\log x}}
 \left(1+\frac{v^3}{p}\right)
 \ll_v (\log x)^{2}.
\]
\end{lemma}

For any real number $z\ge2$, let
\[
 P(z)=\prod_{p\le z}p.
\]

\begin{lemma}\label{lem:rough}
Assume that $\sigma=r_1^{-1}+\cdots+r_t^{-1}=1$. Let $q$ be a fixed positive integer and let
$0<\rho<1$. Suppose that $X$ is a sufficiently large integer,
$3\le Y\le X$, and $T\in\Z$. For any integer $X\le n<2X$, let
\[
 U(n)=|\{a\in\A\cap[1,Y]:\gcd(n+T-a,P(X^\rho))=1\}|.
\]
Then
\begin{equation}\label{eq:roughmoment}
 \sum_{X\le n<2X}U(n)^q\ll_{q,\rho}X.
\end{equation}
\end{lemma}

\begin{proof}
If $3\le Y<Y_0$, then
\[
 U(n)\le \A(Y_0),
\]
and hence
\[
 \sum_{X\le n<2X}U(n)^q
 \le X\A(Y_0)^q
 \ll_q X.
\]
Thus we may assume that $Y$ is a sufficiently large integer.

Let $z=X^\rho$ and $M=\A(Y)$. If $M=0$, then $U(n)=0$ and so \eqref{eq:roughmoment} is immediate.
For $a\in\A\cap[1,Y]$, write
\[
 \mathbf 1_a(n)=\mathbf 1_{\gcd(n+T-a,P(z))=1}.
\]
Then
\[
 U(n)=\sum_{a\in\A\cap[1,Y]}\mathbf 1_a(n)
\]
and
\[
 \sum_{X\le n<2X}U(n)^q=\sum_{a_1,\ldots,a_q\in\A\cap[1,Y]}\sum_{X\le n<2X}\mathbf 1_{a_1}(n)\cdots\mathbf 1_{a_q}(n).
\]
For a $q$-tuple $(a_1,\ldots,a_q)$, if there are $1\le s\le q$ distinct numbers $b_1,\ldots,b_s$, then
\[
 \mathbf 1_{a_1}(n)\cdots\mathbf 1_{a_q}(n)=\mathbf 1_{b_1}(n)\cdots\mathbf 1_{b_s}(n).
\]
Fix $1\le s\le q$ and choose pairwise distinct
$b_1,\ldots,b_s\in\A\cap[1,Y]$. Let
\begin{align}
 S(b_1,\ldots,b_s)&=
 \sum_{X\le n<2X}\mathbf 1_{b_1}(n)\cdots\mathbf 1_{b_s}(n)\notag\\
 &=\left|\left\{X\le n<2X:
 \gcd\!\left(\prod_{i=1}^s(n+T-b_i),P(z)\right)=1\right\}\right|.
\end{align}
Then
\begin{equation}\label{eq:U}
 \sum_{X\le n<2X}U(n)^q\ll_q\sum_{s=1}^{q}\sum_{\substack{b_1<\cdots<b_s\\b_i\in \A\cap[1,Y]}}S(b_1,\ldots,b_s).
\end{equation}
Next, we will obtain an upper bound for $S(b_1,\ldots,b_s)$.

For each prime $p$, let
\[
 I_p=\{n\in\{0,\ldots,p-1\}:n+T-b_i\equiv0\pmod p\text{ for some }1\le i\le s\},
 \qquad
 \nu_p=|I_p|.
\]
If $\nu_p=p$ for some $p\le z$, then $S(b_1,\ldots,b_s)=0$. If $\nu_p<p$ for all primes $p\le z$,
then $\nu_p\le \min\{s,p-1\}$. By the upper bound sieve for residue classes (see Ford \cite[Exercise 2.2]{FordSieve}), we have
\begin{equation}\label{eq:sieveV}
 S(b_1,\ldots,b_s)
 \ll_s
 X V(z),
 \qquad
 V(z)=\prod_{p\le z}
 \left(1-\frac{\nu_p}{p}\right).
\end{equation}
Write
\[
V(z)=\prod_{p\le z}\left(1-\frac1p\right)^s
 \prod_{p\le z}
 \left(1-\frac{\nu_p}{p}\right)
 \left(1-\frac1p\right)^{-s}.
 \]

If $s=1$, then $\nu_p=1$ and so
\[
S(b_1)\ll_{\rho} XV(z)=X\prod_{p\le z}
 \left(1-\frac{1}{p}\right)\ll_{s,\rho}\frac{X}{\log X}.
\]

Now we assume that $s\ge 2$. If $p>2s$ and $p\nmid b_i-b_j$ for any $1\le i<j\le s$, then $n+T-b_1,\ldots,n+T-b_s$ are pairwise distinct modulo $p$, and so
$\nu_p=s$. It follows from Bernoulli's inequality that
\[
 1-\frac{\nu_p}{p}
 =1-\frac{s}{p}
 \le \left(1-\frac1p\right)^s.
\]
Therefore
\[
 \left(1-\frac{\nu_p}{p}\right)
 \left(1-\frac1p\right)^{-s}
 \le1.
\]
If $p>2s$ and $p\mid b_i-b_j$ for some $1\le i<j\le s$, then
\[
 \left(1-\frac{\nu_p}{p}\right)
 \left(1-\frac1p\right)^{-s}
 \le
 \left(1-\frac1p\right)^{-s}
 \le 1+\frac{C_s}{p},
\]
where $C_s>0$ depends only on $s$. For the finitely many primes
$p\le2s$, since $\nu_p<p$, we also have
\[
 \prod_{p\le2s}
 \left(1-\frac{\nu_p}{p}\right)
 \left(1-\frac1p\right)^{-s}
 \ll_s1.
\]
It follows that
\[
\begin{aligned}
 V(z)
 &=
 \prod_{p\le z}\left(1-\frac1p\right)^s
 \prod_{p\le z}
 \left(1-\frac{\nu_p}{p}\right)
 \left(1-\frac1p\right)^{-s} \\
 &\ll_s
 \prod_{p\le z}\left(1-\frac1p\right)^s
 \prod_{1\le i<j\le s}
 \prod_{\substack{p\mid b_i-b_j\\p>2}}
 \left(1+\frac{C_s}{p}\right).
\end{aligned}
\]
By Mertens' theorem and $z=X^\rho$,
\[
 \prod_{p\le z}\left(1-\frac1p\right)^s
 \ll_s\frac1{(\log z)^s}
 \ll_{s,\rho}\frac1{(\log X)^s}.
\]
Then
\[
 V(z)\ll_{s,\rho}\frac1{(\log X)^s}
 \prod_{1\le i<j\le s}
 \prod_{\substack{p\mid b_i-b_j\\p>2}}
 \left(1+\frac{C_s}{p}\right).
\]
Following Erd\H{o}s' idea \cite[Theorem 2]{Erdos}, and applying the inequality of the geometric and arithmetic means
\[
 x_1x_2\cdots x_m \le \frac{x_1^m+\cdots+x_m^m}{m},
\]
we have
\[
 \prod_{1\le i<j\le s}
 \prod_{\substack{p\mid b_i-b_j\\p>2}}
 \left(1+\frac{C_s}{p}\right)
 \le \frac{1}{\binom{s}{2}} \sum_{1\le i<j\le s}
 \prod_{\substack{p\mid b_i-b_j\\p>2}}
 \left(1+\frac{C_s}{p}\right)^{\binom{s}{2}}.
\]
It is clear that there exists a positive integer $v=v(s)$ such that $v^3\ge(1+C_s)^{\binom{s}{2}}-1$ and
\[
\left(1+\frac{C_s}{p}\right)^{\binom{s}{2}}=1+\sum_{k=1}^{\binom{s}{2}}\binom{\binom{s}{2}}{k}\frac{C_s^k}{p^k}\le 1+\frac{(1+C_s)^{\binom{s}{2}}-1}{p}\le1+\frac{v^3}{p}.
\]
Hence,
\begin{equation}\label{eq:Vz-direct}
 V(z)\ll_{s,\rho}\frac1{(\log X)^s}
 \sum_{1\le i<j\le s}
 \prod_{\substack{p\mid b_i-b_j\\p>2}}
 \left(1+\frac{v^3}{p}\right).
\end{equation}
Combining \eqref{eq:sieveV} and \eqref{eq:Vz-direct}, we obtain
\begin{equation}\label{eq:tupleupper}
 S(b_1,\ldots,b_s)
 \ll_{s,\rho}\frac{X}{(\log X)^s}
 \sum_{1\le i<j\le s}
 \prod_{\substack{p\mid b_i-b_j\\p>2}}
 \left(1+\frac{v^3}{p}\right).
\end{equation}
It follows from \eqref{eq:U} that
\begin{align*}
\sum_{X\le n<2X}U(n)^q
&\ll_q \sum_{b_1\in \A\cap[1,Y]}S(b_1)+\sum_{s=2}^{q}\sum_{\substack{b_1<\cdots<b_s\\b_i\in \A\cap[1,Y]}}S(b_1,\ldots,b_s)\\
&\ll_{q,\rho}\frac{XM}{\log X}+\sum_{s=2}^{q}\frac{X}{(\log X)^s}\sum_{1\le i<j\le s}\sum_{\substack{b_1<\cdots<b_s\\b_i\in \A\cap[1,Y]}}
 \prod_{\substack{p\mid b_i-b_j\\p>2}}
 \left(1+\frac{v^3}{p}\right)\\
&\ll_{q,\rho}\frac{XM}{\log X}+\sum_{s=2}^{q}\frac{XM^{s-2}}{(\log X)^s}\sum_{\substack{b<b'\\b,b'\in \A\cap[1,Y]}}
 \prod_{\substack{p\mid b-b'\\p>2}}
 \left(1+\frac{v^3}{p}\right).
\end{align*}
Since $1\le b'-b<Y$, it follows that the number of distinct prime divisors of
$b'-b$ is at most $\log Y/\log2$, and so
\[
 \prod_{\substack{p\mid b-b'\\p>\log Y}}
 \left(1+\frac{v^3}{p}\right)
 \le
 \left(1+\frac{v^3}{\log Y}\right)^{\log Y/\log2}
 \ll_v1.
\]
By Lemma \ref{lem:pair}, we have
\[
\sum_{X\le n<2X}U(n)^q
\ll_{q,\rho}\frac{XM}{\log X}+\sum_{s=2}^{q}\frac{XM^{s-2}(\log Y)^{2}}{(\log X)^s}.
\]
By Lemma \ref{lem:Acount} and $\sigma=1$, we have $M\ll \log Y\le \log X$. Hence,
\[
\sum_{X\le n<2X}U(n)^q
\ll_{q,\rho}X.
\]
This completes the proof of Lemma \ref{lem:rough}.
\end{proof}

A set $\{h_1,\ldots,h_k\}$ of distinct integers is called \emph{admissible} if,
for every prime $p$, the residues $-h_1,\ldots,-h_k$ do not cover all
residue classes modulo $p$.

\begin{lemma}\label{lem:cluster}
For every positive integer $m$, there exist a positive integer $k$
and positive constants $\rho_0,c,X_0$, all depending only on $m$,
such that at least $cX/(\log X)^k$ integers $u\in[X,2X)$ satisfy
\[
 |\{i:1\le i\le k,\ u+h_i\in\PP\}|\ge m
\]
and
\[
 \gcd(u+h_i,P(X^{\rho_0}))=1\qquad(1\le i\le k),
\]
where $X$ is any integer with $X\ge X_0$ and
$\{h_1,\ldots,h_k\}\subseteq[1,X^{1/10}]$ is any admissible
$k$-element set of integers.
\end{lemma}

\begin{proof}
This lemma follows from Maynard's result \cite[Theorem~3.1 and its proof]{Maynard}.
For completeness, we give the details of the application. Let
\[
 L_i(u)=u+h_i\qquad(1\le i\le k),\qquad
 \mathcal L=\{L_1,\ldots,L_k\}.
\]
In the notation of \cite[Hypothesis~1]{Maynard}, let
\[
 (A,\mathcal L,P,B,x,\theta)
 =\left(\N,\mathcal L,\PP,1,X,\frac13\right).
\]
The admissibility of $\{h_1,\ldots,h_k\}$ implies that $\mathcal L$
is admissible. We first verify  \cite[Hypothesis~1]{Maynard}, uniformly for
$1\le h_i\le X^{1/10}$.

Since $X$ is an integer, it follows that
\[
 |\{X\le u<2X:u\equiv a\pmod q\}|=\frac Xq+O(1).
\]
Thus,
\[
 \sum_{q\le X^{1/3}}\max_{0\le a<q}
 \left||\{X\le u<2X:u\equiv a\pmod q\}|-\frac Xq\right|
 \ll X^{1/3}\ll\frac{X}{(\log X)^{100k^2}},
\]
and, for any $q\le X^{1/3}$,
\[
 |\{X\le u<2X:u\equiv a\pmod q\}|
 \le\frac Xq+1\ll\frac{X}{q}.
\]
These are parts (1) and (3) of Hypothesis~1.

For any nonnegative integer $h$, let
\[
 \pi_h=|\PP\cap[X+h,2X+h)|,\qquad
 \pi_h(q,b)=|\{p\in\PP\cap[X+h,2X+h):p\equiv b\pmod q\}|.
\]
By the Bombieri--Vinogradov theorem,
\[
 \sum_{q\le X^{1/3}}\max_{(b,q)=1}
 \left|\pi_0(q,b)-\frac{\pi_0}{\varphi(q)}\right|
 \ll_k\frac{X}{(\log X)^{100k^2+2}}.
\]
Clearly, $|\pi_h(q,b)-\pi_0(q,b)|\le 2h$ and $|\pi_h-\pi_0|\le 2h$. Thus, by the prime number theorem,
\begin{align*}
 \sum_{q\le X^{1/3}}\max_{(b,q)=1}
 \left|\pi_h(q,b)-\frac{\pi_h}{\varphi(q)}\right|
 &\ll_k\frac{X}{(\log X)^{100k^2+2}}+hX^{1/3}\\
 &\ll_k\frac{X}{(\log X)^{100k^2+2}}\\
 &\ll\frac{\pi_h}{(\log X)^{100k^2}},
\end{align*}
uniformly for $1\le h\le X^{1/10}$. This is part (2) of Hypothesis~1.

For $\alpha=1/10$, the coefficient and $B$ bounds in
\cite[Theorem~3.1]{Maynard} hold, since $a_i=1$,
$1\le b_i=h_i\le X^{1/10}$, and $B=1$. Take $\delta=1/2$. Then
\[
 \frac1k\sum_{i=1}^k\pi_{h_i}\ge\delta\frac{X}{\log X}
\]
for all sufficiently large $X$, which is the condition
in \cite[Theorem~3.1]{Maynard}. Let $c_1>0$ be the constant in
\cite[equation~(6.10)]{Maynard}. Choose $k=k(m)$ sufficiently large to
satisfy
\begin{equation}\label{eq:condition610}
 \delta>(\log k)^{-1},\qquad
 3c_1\delta\log k-2m\ge1.
\end{equation}
Following the proof of \cite[Theorem~3.1]{Maynard}, take
\[
 \rho=\frac{c_0}{k^3\log k},
\]
where $c_0>0$ is the sufficiently small constant chosen there.
For a sufficiently large integer $X$, we have
\[
 k\le(\log X)^{1/10},\qquad
 \frac{k(\log\log X)^2}{\log X}\le\rho\le\frac{\theta}{10}=\frac1{30}.
\]
The first inequality also implies $k\le(\log X)^{1/5}$, as required
in the proof of \cite[Theorem~3.1]{Maynard}.

Let $\lambda_u$ be the non-negative weights whose existence is guaranteed by \cite[Proposition~6.1]{Maynard} with $R=X^{\theta/3}=X^{1/9}$, and let
\[
 F(u)=\sum_{i=1}^k\mathbf1_{\PP}(u+h_i)-m
       -k\sum_{i=1}^k\sum_{\substack{p\mid u+h_i\\p<X^\rho}}1,
 \qquad
 S=\sum_{X\le u<2X}F(u)\lambda_u.
\]
Then $S$ is precisely the sum in
\cite[equation~(6.2)]{Maynard}. In the proof of \cite[Theorem~3.1]{Maynard}, Maynard noted that for
any $X\le u <2X$, $F(u)$ is positive only if at least $m+1$ of the $u+h_1,\ldots,u+h_k$ are
primes in $\PP$, and none of the $u+h_1,\ldots,u+h_k$ have any prime factors less than $X^\rho$. It follows from \eqref{eq:condition610} and \cite[equation~(6.10)]{Maynard} that there exists an integer $u\in[X,2X)$ such that $F(u)$ is positive. Note that $P(z)=\prod_{p\le z}p$. Take $\rho_0=\rho/2$ and so this integer $u$ satisfies
\[
 |\{i:1\le i\le k,\ u+h_i\in\PP\}|\ge m
\]
and
\[
 \gcd(u+h_i,P(X^{\rho_0}))=1\qquad(1\le i\le k).
\]
Next, we will show that there are at least $cX/(\log X)^k$ integers $u\in[X,2X)$ such that $F(u)$ is positive.

Let $\mathcal G=\{u\in\Z:X\le u<2X,\ F(u)>0\}$. It follows from
\cite[equations~(6.10)--(6.11)]{Maynard} that
\[
 S\gg_m X(\log X)^k.
\]
It follows from \cite[equation~(6.3)]{Maynard} that
\[
 \lambda_u\ll_{m}(\log X)^{2k}
 \qquad(u\in\mathcal G).
\]
Since $F(u)\le k$ and $\lambda_u\ge0$,
it follows that
\[
 X(\log X)^k\ll_m S
 \le\sum_{u\in\mathcal G}F(u)\lambda_u
 \ll_m|\mathcal G|(\log X)^{2k}.
\]
Hence
\[
 |\mathcal G|\gg_m\frac{X}{(\log X)^k}.
\]
This completes
the proof of Lemma~\ref{lem:cluster}.
\end{proof}

\section{Proof of the main theorem}

\begin{proof}[Proof of Theorem~\ref{thm:main}]
For every positive integer $m$, let $k$ and $\rho_0$ be given by
Lemma~\ref{lem:cluster}, and let
\[
 0<\rho<\min\{\rho_0,1/2\}.
\]
For a sufficiently large integer $X$, let
\[
 Y=X^{1/10},\qquad T=\lfloor Y\rfloor+1,\qquad
 \mathcal B=\A\cap[1,Y],\qquad W=\prod_{p\le k}p.
\]
By Lemma~\ref{lem:Acount} and $\sigma=1$, we have $|\mathcal B|\asymp\log X$. Let $\mathcal D$ be a largest subset of $\mathcal B$ such that all elements are congruent modulo $W$. Then
\begin{equation}\label{eq:Dsize}
 |\mathcal D|\ge\frac{|\mathcal B|}{W}\gg\log X.
\end{equation}
It is clear that, for every $k$-element subset $H=\{a_1,\ldots,a_k\}\subseteq\mathcal D$, the set
$\{T-a_1,\ldots,T-a_k\}$ is admissible, and
\[
 1\le T-a_i\le Y=X^{1/10}.
\]

For any integer $X\le u<2X$, let $N(u)$ be the number of $k$-element subsets
$H\subseteq\mathcal D$ satisfying
\[
 |\{a\in H:u+T-a\in\PP\}|\ge m
\]
and
\[
 \gcd(u+T-a,P(X^\rho))=1\qquad(a\in H).
\]
Since $\rho<\rho_0$, it follows from Lemma~\ref{lem:cluster} and \eqref{eq:Dsize} that
\begin{equation}\label{eq:Nfirst}
 \sum_{X\le u<2X}N(u)
 \gg\binom{|\mathcal D|}{k}\frac{X}{(\log X)^k}
 \gg X.
\end{equation}
Let
\[
 U(u)=|\{a\in\mathcal B:\gcd(u+T-a,P(X^\rho))=1\}|.
\]
Then
\[
 N(u)\le\binom{U(u)}k\le U(u)^k.
\]
By Lemma~\ref{lem:rough} with $q=2k$,
\begin{equation}\label{eq:Nsecond}
 \sum_{X\le u<2X}N(u)^2
 \le\sum_{X\le u<2X}U(u)^{2k}\ll X.
\end{equation}
By \eqref{eq:Nfirst}, \eqref{eq:Nsecond} and the Cauchy--Schwarz inequality,
\[
 |\{X\le u<2X:N(u)>0\}|
 \ge\frac{\left(\sum_{X\le u<2X}N(u)\right)^2}
            {\sum_{X\le u<2X}N(u)^2}
 \gg X.
\]
If $N(u)>0$, then $r(u+T)\ge m$. Hence,
\[
 |\{n\le3X:r(n)\ge m\}|\ge |\{X\le u<2X:r(u+T)\ge m\}|\gg X.
\]
This completes the proof of Theorem~\ref{thm:main}.
\end{proof}

\noindent\textbf{Acknowledgments.} This work is supported by the National Natural Science Foundation of China (Grant
No. 12371005).

\end{document}